\documentclass{amsart}
\usepackage{amsthm}
\usepackage{amsmath}
\usepackage{amsfonts}
\usepackage{amssymb} 
\usepackage{amscd}
\usepackage[all,cmtip]{xy}
\newtheorem{theorem}{Theorem}
\newtheorem{proposition}[theorem]{Proposition}
\newtheorem{lemma}[theorem]{Lemma}
\newtheorem{corollary}[theorem]{Corollary}
\newtheorem{conjecture}{Conjecture}

\begin{document}
\title{Iwasawa Invariants of Some Non-Cyclotomic $\mathbb Z_p$-extensions}
\author{David Hubbard}
\address{35 Holt Circle,  Hamilton, NJ 08619}
\email{dhubbard@erols.com} 
\author{Lawrence C. Washington}
\address{Department of Mathematics,
University of Maryland,
College Park, MD 20742}
\email{lcw@math.umd.edu}

\dedicatory{To the memory of Professor Kenkichi Iwasawa on the 100th anniversary of his birth.}
\begin{abstract}Iwasawa showed that there are non-cyclotomic $\mathbb Z_p$-extensions with
positive $\mu$-invariant. We show that these $\mu$-invariants can be evaluated explicitly in many situations
when $p=2$ and $p=3$.\end{abstract}
\maketitle

\section{Introduction}
Let $p$ be prime, let $K$ be a number field, and let $K_{\infty}/K$ be a $\mathbb Z_p$-extension of $K$.
Let $K_n$ be the subfield of $K_{\infty}$ that has degree $p^n$ over $K$ and let $p^{e_n}$
be the power of $p$ dividing the class number of $K_n$.
A well-known result of Iwasawa says that there exist integers $\lambda, \mu, \nu$, independent of $n$, 
such that
$$
e_n=\mu p^n +\lambda n + \nu
$$
for all sufficiently large $n$. When $K_{\infty}$ is the cyclotomic $\mathbb Z_p$-extension of $K$,
it is conjectured that $\mu=0$, and this has been proved \cite{FW} when $K/\mathbb Q$ is abelian.
In \cite{Iw}, Iwasawa showed that there exist non-cyclotomic $\mathbb Z_p$-extensions with $\mu>0$.
A natural question is what are the actual values of the Iwasawa invariants for these $\mathbb Z_p$-extensions?
In the following, we give examples of Iwasawa's construction when $p=3$ and $p=2$, and we show that in many cases it is possible to evaluate
the invariants $\lambda, \mu, \nu$ explicitly.

\section{Preliminaries}\label{SectPrelim}

For reference, here is the basic notation that will be used throughout the paper.
\begin{itemize}
\item $k_0$ is an imaginary quadratic field. Usually, $k_0=\mathbb Q(\sqrt{-1})$ or $\mathbb Q(\sqrt{-3})$.

\item $p$ is a prime number that is non-split in $k_0/\mathbb Q$.

\item $
k_0\subset k_1 \subset \cdots k_n \subset \cdots \subset k_{\infty}
$
is the anticyclotomic $\mathbb Z_p$-extension of $k_0$. This means that $\text{Gal}(k_0/\mathbb Q)$ acts by $-1$ on $\text{Gal}(k_{\infty}/k_0)$. 

\item $K_0/k_0$ is a Galois extension of degree $p$ such that $K_0$ is Galois over $\mathbb Q$ and
$$
K_n = K_0k_n.
$$
Therefore, $K_{\infty}/K_0$ is a $\mathbb Z_p$-extension. In this situation,
the complex conjugation that generates $\text{Gal}(k_0/\mathbb Q)$ can be lifted to
an automorphism $\sigma$ of order 2 of $K_{\infty}$ such that $\sigma \gamma\sigma = \gamma^{-1}$
for all $\gamma\in \text{Gal}(K_{\infty}/K_0)$.

\item $A_n$ is the $p$-Sylow subgroup of the ideal class group of $K_n$ and 
$$
X=\lim_{\leftarrow} A_n
$$
is the inverse limit of these groups with respect to the norm.

\item $\Lambda=\mathbb Z_p[[T]]$ and $\omega_n=(1+T)^{p^n}-1\in \Lambda$.
It is a standard fact that $X$ is a finitely generated $\Lambda$-torsion module.
\item $
\nu_n = \omega_n(T)/T = (1+T)^{p^n-1} + (1+T)^{p^n-2} +\cdots + (1+T) + 1.
$

\item $s=$ the number of primes that are inert in $k_0/\mathbb Q$ and ramify in $K_0/k_0$.

\item We use the notation $A_1 = 27\times 9^2\times 3^7$, for example, to indicate that $A_1$ is a product of a cyclic group
of order 27, two cyclic groups of order 9, and seven groups of order 3.
\end{itemize}
\medskip

A crucial fact is the following.
\begin{lemma}\label{LemmaSplit} Let $q\ne p$ be a prime that is inert in $k_0/\mathbb Q$. Then $q$ splits completely in $k_{\infty}/k_0$.
\end{lemma}
\begin{proof} (cf. \cite[Lemma 10]{washthesis})
Let $\sigma\in \text{Gal}(k_{\infty}/\mathbb Q)$ have order 2 (for example, complex conjugation)
and let $\Gamma=\text{Gal}(k_{\infty}/k_0)$. Then $\text{Gal}(k_{\infty}/\mathbb Q) = \Gamma\cup \sigma\Gamma$, 
and $\sigma\gamma\sigma^{-1} = \gamma^{-1}$ for all $\gamma\in \Gamma$.

We are assuming that $q$ is inert in $k_0/\mathbb Q$ and that $q$ is unramified in $k_{\infty}/k_0$.
Fix a prime ${\mathcal{P}}$ of $k_{\infty}$ dividing $q$ and let $D\subseteq \text{Gal}(k_{\infty}/\mathbb Q)$ 
be the decomposition group for ${\mathcal{P}}$.  Since $q$ is unramified,
$D$ is procyclic.
Since $q$ is inert in $k_0$, which is the fixed field of $\Gamma$, it follows that $D$ is not contained in $\Gamma$.
Therefore, $D$ contains a topological generator not in $\Gamma$, say $\sigma\gamma$ with $\gamma\in \Gamma$.  
But $(\sigma\gamma)^2 = 1$, so $D$ has order 2. 
Therefore, $D\cap \Gamma=1$, which means that
$q$ splits completely in $k_{\infty}/k_0$. \end{proof}

\section{Iwasawa's Construction}

\begin{theorem}\label{Iw}
Suppose $s$ distinct primes $q\ne p$ are inert in $k_0/\mathbb Q$ and ramify in $K_0/k_0$.
 Then $\mu\ge s-1$ for the $\mathbb Z_p$-extension
$K_{\infty}/K_0$.
\end{theorem}
\begin{proof}
We need the following result of Chevalley \cite{Ch} (see also \cite{Lem}):
\begin{proposition}\label{Chev} Let $L/K$ be a cyclic extension of degree $n$ of number fields, let $G=\text{Gal}(L/K)$, let $h$ be the class number of $K$,
let $C^{G}$ be the set of ideal classes
of $L$ that are fixed by $G$, let $N$ be the norm from $L$ to $K$,  let $E$ denote the group of
units of the ring of integers of $K$, let $S$ be the set of
primes of $K$ that ramify in $L/K$, and let $e_{\mathcal{P}}$ be the ramification index of such a prime. Then
$$
|C^G| = \frac{h \prod_{{\mathcal{P}}\in S} e_{\mathcal{P}} }{n[E: E\cap N(L^{\times})]}.
$$
\end{proposition}

We apply this result to $K_n/k_n$.
 By Lemma \ref{LemmaSplit}, each prime $q$
splits completely in $k_n/k_0$ and the primes over $q$ ramify in $K_n/k_n$. Therefore, there are at least $s\times p^n$ primes of 
$k_n$ that ramify in $K_n/k_n$. This means that $p^{sp^n}$ divides  $\prod e_{\mathcal{P}}$ in the numerator of Chevalley's formula for $K_n/k_n$.
Since $k_n$ is a totally complex extension of $\mathbb Q$ of degree $2\times p^n$, the unit group (including the roots of unity) mod $p$th powers
has $p^n$ or $p^n-1$ generators, depending on whether $k_0$ contains or does not contain a $p$th root of unity. 
But the $p$th power of a unit of $k_n$ is automatically a norm from $K_n$. Therefore, $[E: E\cap N(K_n^{\times})]\mid p^{p^n}$.
Putting everything together, we find that 
$$
 p^{(s-1)p^n-1}\mid |C^G|,
$$
where  $C^G$ is the group of fixed ideal classes $C^G$ of $K_n$. Let $p^{e_n}$ be the power of $p$ dividing the class number of $K_n$.
Then $e_n\ge (s-1)p^n-1$, which yields the result.  \end{proof}

In Section \ref{SectComp}, we give 
explicit values of $\mu$ for some choices of $K_0$. 

A natural question is whether the estimate $\mu\ge s-1$ is sharp. In the examples in Section \ref{SectComp} for which we are able to evaluate $\mu$
exactly, the value is always $s-1$, but our methods will never give the exact value if $\mu > s-1$, so this is not evidence. We found
one example for $p=2$ ($d=3\cdot 17\cdot 19$) for which we suspect $\mu=s$. In Section \ref{SectAmbig},
we discuss what might cause $\mu$ to be larger than $s-1$.

\section{Initial Layers}

\subsection{$p=3$}

For computations, we need to know $k_1$ explicitly. The following lemma is contained in \cite{Oh-Kim} and \cite{VH}, 
but we include the proof for convenience.
\begin{lemma}\label{LemmaKummer}   Let $k_0=\mathbb Q(\sqrt{-3})$ and let $k_0\subset k_1 \cdots \subset k_{\infty}$
be the anticyclotomic $\mathbb Z_3$-extension of $k_0$. Then $k_1=k_0(\sqrt[3]{3})$.\end{lemma}
\begin{proof} Kummer theory says that there exists $\beta\in k_0$ such that $k_1=k_0(\sqrt[3]{\beta})$. 
The Kummer pairing
$$
\text{Gal}(k_1/k_0) \times \langle \beta\rangle (k_0^{\times})^3/(k_0^{\times})^3 \rightarrow \mu_3
$$
satisfies 
\begin{align*}
\langle \gamma, \sigma(\beta)\rangle^{-1}&=\langle \gamma^{-1}, \sigma(\beta)\rangle=\langle \gamma^{\sigma}, \sigma(\beta)\rangle\\
 &= \sigma \langle \gamma, \beta \rangle = \langle \gamma, \beta \rangle^{-1}
\end{align*}
for all $\gamma\in \text{Gal}(k_1/k_0)$. The non-degeneracy of the pairing implies that 
$\sigma(\beta) = \beta \alpha^3$ for some $\alpha\in k_0$. Taking the norm to $\mathbb Q$ of both sides yields $N(\alpha)=1$.
Hilbert's Theorem 90 says $\alpha=\delta^{1-\sigma}$ for some $\delta\in k_0$, which yields $\sigma(\beta\delta^3) = \beta\delta^3$.
Therefore, we may modify $\beta$ by a cube and assume it is rational, and in  fact a positive cube-free integer. 
Since a $\mathbb Z_3$-extension is unramified outside 3, the generator $\beta$ 
cannot have prime factors other than 3, so we can take $\beta=3$.\end{proof}

Van Huele \cite{VH} also shows that $X^9+9X^6+27X^3+3$ generates $k_2/k_0$ . He also gives a polynomial generating $k_3/k_0$.

\subsection{$p=2$}

\begin{proposition} Let $k_0=\mathbb Q(i)$ and let $k_0\subset k_1 \cdots \subset k_{\infty}$
be the anticyclotomic $\mathbb Z_2$-extension of $k_0$. Then
$$
k_1=k_0(\sqrt{-2})=\mathbb Q(\zeta_8), \quad k_2= k_0(\sqrt[4]{-2}), \quad k_3 = k_0(\sqrt[8]{-2}).
$$
\end{proposition}
\begin{proof}
Of course, the statements about $k_1$ and $k_2$ follow from the one about $k_3$, but we need to prove them
in order to obtain $k_3$.

First, $k_1$ is Galois over $\mathbb Q$ of degree 4, is unramified outside 2, and contains $\mathbb Q(i)$, so we must have 
$k_1=\mathbb Q(\zeta_8)$.

Write $k_2=k_0(\sqrt[4]{\alpha})$ for some $\alpha\in k_0$. Since $k_2/k_0$ is unramified outside 2,
we may assume that $\alpha=i^a(1+i)^b$ for some $a, b$ that we can take mod 4. Because $k_2/\mathbb Q$ is Galois,
$\sqrt[4]{\overline{\alpha}}$ must generate the same extension, so either $\alpha/\overline{\alpha}$
or $\alpha\overline{\alpha}$ is a fourth power in $k_0$.

If $\alpha\overline{\alpha}=2^b$ is a fourth power, then $b$ is a multiple of 4, so we are reduced to considering
$\alpha=i^a$. But this yields $k_2\subseteq \mathbb Q(\zeta_{16})$, which contradicts the fact that $k_2/\mathbb Q$
is non-abelian.

If $\alpha/\overline{\alpha}= i^{2a+b}$ is a fourth power, then $2a+b\equiv 0\pmod 4$, so we may assume that
$b=2a$. This yields $\alpha = (-2)^a$. Since $a=0$ and $a=2$ do not yield extensions of degree 4,
we must have $a=1$ or $3$, both of which give $k_2=k_0(\sqrt[4]{-2})$.

Write $k_3=k_1(\sqrt[4]{\beta})$ for some $\beta\in k_1$. Because $k_3/k_1$ is unramified outside 2, and
every unit of $\mathbb Z[\zeta_8]$ is of the form $\zeta_8^a(1+\sqrt{2})^b$ (a root of unity times a unit 
from the maximal real subfield), we may assume that
$$
\beta=\zeta_8^a(1+\sqrt{2})^b(1-\zeta_8)^c.
$$
Since $k_3/\mathbb Q$ is Galois, we must have $\beta\overline{\beta}$ or $\beta/\overline{\beta}$ equal to
a fourth power in $k_1$. 

If $\beta\bar{\beta}= (1+\sqrt{2})^{2b-1}(\sqrt{2})^{c}$ is a fourth power, then 2-adic valuations show
that $c$ must be even. If $\beta/\overline{\beta}=\zeta_8^{2a-3c}$ is a fourth power, then $c$ is even.
Therefore, in both cases we have that $c$ is even. Since $(1-\zeta_8)^2/\sqrt{2}$ is a unit, we can start over
and write
$$
\beta=\zeta_8^e(1+\sqrt{2})^f (\sqrt{2})^g.
$$

Suppose $\beta\overline{\beta}= (1+\sqrt{2})^{2f}(\sqrt{2})^{2g}$ is a fourth power in $k_1$. 
Taking the norm to $k_0$ yields that $2^{2g}$ is a fourth power in $k_0$, so $g$ is even, hence $(1+\sqrt{2})^{2f}$
is a fourth power.
But $1+\sqrt{2}$ is the fundamental unit of $\mathbb Z[\zeta_8]$, so $f$ is even.
But if both $f$ and $g$ are even, $k_2=k_1(\sqrt{\beta})\subseteq \mathbb Q(\zeta_{16})$, which is impossible
since $k_2/\mathbb Q$ is non-abelian. 

Therefore, $\beta/\overline{\beta} = \zeta_8^{2e}$ is a fourth power in $k_1$, which implies that $e$ is even.
Since $-1=\zeta_8^4\in k_1$, we may assume that $\zeta_8^e= 1$ or $i$.
Since $k_1(\sqrt[4]{-2})= k_2=k_1(\sqrt{\beta})$, it follows that $\sqrt{-2}\beta$ is a square in $k_1$.
Taking the norm to $k_0$ yields that $\pm 2^{g+1}$ is a square in $k_0$, so $g$ is odd. 

Let $\beta'$ be the conjugate of $\beta$ under the map $\sqrt{2}\mapsto -\sqrt{2}$, $i\mapsto i$. 
Since $k_3/\mathbb Q(\sqrt{2})$ is Galois, we must have either $\beta\beta'= \pm 2^g$ a fourth power in $k_1$ or 
$\beta/\beta'=\pm (1+\sqrt{2})^{2f}$ a fourth power in $k_1$. The first case is impossible since $g$ is odd.
The second case implies that $f$ is even, since $1+\sqrt{2}$ is a fundamental unit.

We now have that $g$ is odd and $f$ is even. Changing $\beta$
to $\beta^3$ if necessary and removing fourth powers means that we can assume $g=1$ and $f=0$ or $2$.
We therefore have the following choices for $\beta$:
$$
\beta_1 = \sqrt{2}, \quad \beta_2 = i\sqrt{2}, \quad \beta_3 = (1+\sqrt{2})^2\sqrt{2}, 
\quad \beta_4 = i(1+\sqrt{2})^2\sqrt{2}.
$$
A calculation in Pari-GP shows that $k_1(\sqrt[4]{\beta_1})=k_1(\sqrt[4]{\beta_4})$ and
 $k_1(\sqrt[4]{\beta_2})=k_1(\sqrt[4]{\beta_3})$, so we need to consider only $\beta_1$ and $\beta_2$.

Suppose $\beta = \beta_1$.   
Let $\sigma$ generate $\text{Gal}(k_3/k_0)$.  
Then $\sigma: \sqrt[8]{2}\mapsto \zeta_8^a\sqrt[8]{2}$ for some $a$. If $a$ is even then $\sigma^4=1$, so we must have $a$ odd. Therefore,
$\sigma(\sqrt{2}) = -\sqrt{2}$, so 
$$
\sigma(\zeta_8) = \sigma((1+i)/\sqrt{2}) = -\zeta_8.
$$
Let $J\in \text{Gal}(k_3/\mathbb Q)$ be complex conjugation 
(we assume that $k_3$ has been embedded into $\mathbb C$).
Then 
$$
J(\sqrt[8]{2}) = \sqrt[8]{2}, \quad J(\zeta_8)=\zeta_8^{-1}.
$$
A calculation shows that $(\sigma J)^2(\sqrt[8]{2}) = -\sqrt[8]{2}$, so $\sigma J$ does not have order 2.
Therefore, $J$ does not act as $-1$ on $\text{Gal}(k_3/k_0)$, contradicting the definition of the anticyclotomic
$\mathbb Z_2$-extension. Therefore, $\beta\ne \beta_1$.

The only remaining possibility is $\beta=\beta_2 = i\sqrt{2}= \sqrt{-2}$. This completes the proof. \end{proof}

Note that $k_1=k_0(\sqrt{i})$ and a calculation shows that $k_2=k_1\big(\sqrt{\zeta_8(1+\sqrt{2})}\big)$. Therefore, both extensions
are obtained by adjoining the square root of a unit. In fact, this holds more generally.

\begin{proposition}\label{sqrtunit} Let $n\ge 0$. Then there exists a unit $u_n \in k_n$ such that $k_{n+1}=k_n(\sqrt{u_n})$.
\end{proposition}
\begin{proof}
 Let $\sigma$ be complex conjugation
(after a complex embedding is fixed).
The extension is generated by $\sqrt{\beta}$ for some $\beta$.
Because $\sigma$ maps $k_n$ and $k_{n+1}$ to themselves, $\sigma \beta = \beta \alpha^2$ for some $\alpha\in k_n$.
Therefore, $(\alpha^{1+\sigma })^2=1$.
Since $x^{1+\sigma} = |x|^2 > 0$ for all $x$, we must have
$\alpha^{1+ \sigma} = 1$. Let $F_n$ be the fixed field of $\sigma$. Hilbert's Theorem 90 applied to $k_n/F_n$
says that $\alpha = \gamma^{1-\sigma}$ for some $\gamma\in k_{n}$,
so $\beta \gamma^2$ is fixed by $\sigma$.
Therefore, we may assume that $\beta$ lies in $F_n$.
Since $k_n/F_n$ is totally ramified at 2, the valuation
of $\beta$ at each prime of $k_n$ above 2 is even. The valuation at all other primes 
must be even because $k_{n+1}/k_n$ is unramified outside 2. Therefore, $(\beta)$ is the square of an ideal
in $k_n$. The class number of $k_n$ is odd, so we may raise
$\beta$ to an odd power and assume that the ideal $(\beta)$ is the square
of a principal ideal in $k_n$. This means that $\beta$ is a square times a unit, which yields the result.\end{proof}

\section{The $\mathbb{\lambda}$ invariant}

In this section, we prove that, in some situations we consider, $\lambda$ is even, and that if $\lambda \ne 0$
then this contributes to $e_1-e_0$ (that is, we do not have to wait until larger $n$ to see a contribution
to $e_n$).

\begin{theorem}\label{PropEven} Assume that $K_{\infty}/K_0$ is ramified at only one prime.
Then $\lambda$ is even for the $\mathbb Z_p$-extension $K_{\infty}/K_0$.
\end{theorem}
\begin{proof} 
If $K_{\infty}/K_0$ is not totally ramified, we can replace $K_{\infty}/K_0$ with $K_{\infty}/K_e$ for some $e>0$ and use the same proof.
Therefore, for simplicity, we assume that $K_{\infty}/K_0$ is totally ramified.

Let $\sigma$ and $\Gamma$ be as in Section \ref{SectPrelim} and assume that $\gamma$ is a topological generator of $\Gamma$.
Then $\sigma$ and $\gamma$ act on $X$, and $\Lambda=\mathbb Z_p[[T]]$ acts via $1+T=\gamma$. 
Let $X^- = \{x\in X\, |\, \sigma x=-x\}$.
Define the homomorphism of $\mathbb Z_p$-modules
$$
f: \, X^-\times X^- \rightarrow X, \qquad (x,y)\mapsto x-\gamma y.
$$
Suppose $(x,y)\in \text{Ker}(f)$. Then $x=\gamma y$. Therefore,
$$
-\gamma y =-x=\sigma x = \sigma\gamma y=\gamma^{-1} \sigma y = -\gamma^{-1}y.
$$
Therefore,
$$
0=(\gamma^2 -1)y = ((1+T)^2-1)y = (2+T)Ty.
$$

Suppose $p$ is odd.
Since $2+T$ is invertible in $\Lambda$, we must have $Ty=0$. Therefore, 
$$
Tx=T\gamma y = \gamma Ty = 0.$$
Therefore, $\text{Ker}(f)\subseteq X[T]\times X[T]$ (where $X[T]$ denotes the kernel of 
multiplication by $T$; we could use $X^-[T]$, but $X^-$ is not necessarily
a $\Lambda$-module and $X$ suffices for our purposes).

If $p=2$, we find that $\text{Ker}(f)\subseteq X[\omega_1(T)]\times X[\omega_1(T)]$.

Let $x_1\in X$. Then
$(1-\sigma)  x_1$ and $(1-\sigma)\gamma x_1$ are in $X^{-}$, and
\begin{align*}
f\left((1-\sigma) x_1, \, (1-\sigma)\gamma x_1\right) &= (1-\sigma) x_1-\gamma (1-\sigma)\gamma x_1\\
&=(1-\gamma^2)x_1.
\end{align*}
If $p$ is odd, we have $(1-\gamma^2)x_1 = T(-2-T)x_1$. Let $Tx\in TX$. Since $-2-T$ is invertible in $\Lambda$, 
we can write $x=(-2-T)x_1$ for some $x_1\in X$, and conclude that $\text{Im}(f)\supseteq TX$.

If $p=2$, we have $\text{Im}(f)\supseteq \omega_1(T)X$.

Since one prime ramifies in $K_{\infty}/K_0$ and it is totally ramified, $X/TX\simeq A_0$, which is finite.
When $p=2$, we use the fact that $X/\omega_1X\simeq A_1$, which is finite.
Therefore, in both cases, $X/\text{Im}(f)$ is finite.

Since $X$ is a finitely generated $\Lambda$-module, the structure theorem implies that the characteristic power
series of $X$ is not a multiple of $T$ or $\omega_1(T)$. Therefore, $X[T]$ and $X[\omega_1(T)]$ are finite, 
so $\text{Ker}(f)$ is finite.

Therefore, there is an exact sequence
$$
0\to \text{finite} \to X^-\times X^-\to X\to \text{finite} \to 0.
$$
Tensoring with $\mathbb Q_p$ yields an isomorphism
$$
\left(X^-\otimes \mathbb Q_p\right)^2 \simeq X\otimes \mathbb Q_p.$$
The structure theorem tells us that the $\mathbb Q_p$-dimension of the right side is $\lambda$, which must therefore
be even. \end{proof}

{\bf Remark.} The above proof is based on a proof of Ren\'e Schoof \cite{Schoof}
that, when $d$ contains no prime factors that are 1 mod 3, there is a group $A$ such that the ideal class group
of $\mathbb Q(\sqrt{-3}, \sqrt[3]{d})$ has the form $A\times A$.

\medskip

In our numerical examples, $k_0$ has class number prime to $p$. In this case, we can make a stronger statement about $\lambda$.
\begin{theorem}\label{LambdaEvenMore} Assume that $p$ does not divide the class number of $k_0$. Then $\lambda$
is divisible by $p-1$ for the $\mathbb Z_p$-extension $K_{\infty}/K_0$. 
\end{theorem}
\begin{proof} Since there is only one prime above $p$ in $k_0$, there is exactly one prime ramifying in $k_n/k_0$, 
and this extension is totally ramified at this prime. Since we are assuming that $k_0$ has class number prime to $p$, 
the class number of $k_n$ is also prime to $p$.

 Let $\tau$ generate $\text{Gal}(K_{\infty}/k_{\infty})$. Then $\tau$ acts on $X$ and $1+\tau+\tau^2+\cdots +\tau^{p-1}$
annihilates $X$ since the class number of each $k_n$ is prime to $p$.
Therefore, $X$ is a module over $\mathbb Z_p[\tau]/(1+\tau+\cdots +\tau^{p-1})\simeq \mathbb Z_p[\zeta_p]$,
where $\zeta_p$ is a primitive $p$th root of unity. Tensoring with $\mathbb Q_p$ shows that
$X\otimes_{\mathbb Z_p} \mathbb Q_p$ is a vector space over $\mathbb Q_p(\zeta_p)$. But
$$
\lambda = \dim_{\mathbb Q_p} \left(X\otimes \mathbb Q_p\right) = (p-1) \dim_{\mathbb Q_p(\zeta_p)}
\left(X\otimes \mathbb Q_p\right).
$$
This proves the theorem.\end{proof}

The following will allow us to show $\lambda=0$ in some  cases. Recall that an elementary $\Lambda$-module
is a module of the form
$$
E=\bigoplus_i\Lambda/(p^{\mu_i})\oplus \bigoplus_j \Lambda/(f_j),
$$
where each $\mu_i$ is a positive integer and each $f_j$ is a distinguished polynomial. (An elementary module can also include a summand
$\Lambda^r$, but we will not need this case.) For such a module,
$\mu=\sum \mu_i$ and $\lambda=\sum \deg f_i$.

\begin{theorem}\label{LambdaThm} Let $E$ be an elementary $\Lambda$-module with $\mu=0$ and $\lambda \ge 2$. Then
$$
\#(E/\nu_1 E)\ge p^m,
$$
where
$$
m = \text{min}(\lambda, \, p-1).
$$
When $p=2$,
$$
\#(E/\nu_2 E)\ge \begin{cases} 16 \text{ if } \lambda = 2\\ 8\text{ if } \lambda \ge 3.
\end{cases}
$$
\end{theorem}

\begin{proof} 
Suppose that $E$ contains a summand $E_1=\Lambda/(f(T))$, where $f(T)$ is distinguished of degree $\lambda$,
with $1\le \lambda \le p-1$.
Write $f(T)=T^{\lambda}+a_{\lambda-1}T^{\lambda-1}+\cdots+a_0$, with $p\mid a_i$ for all $i$.
Then
\begin{align*}
(\nu_1(T), \, f(T)) &= (\nu_1(T)-T^{p-1-\lambda}f(T), \, f(T))\\
&\subseteq (p,   \, f(T)),
\end{align*}
which has index $p^{\lambda}$ in $\Lambda$. Therefore, 
$$\#(E_1/\nu_1E_1)\ge p^{\lambda}.$$

Now suppose $\deg f(T)=\lambda \ge p$. The division algorithm for distinguished polynomials lets us write
\begin{equation}\label{divalg}
f(T)=\nu_1(T) q(T) + r(T)
\end{equation}
with $\deg r(T)\le p-2$. Reduce (\ref{divalg}) mod $p$. The fact that $f(T) \mod p$ 
and $\nu_1(T) \mod p$ are monomials and the uniqueness in the division algorithm mod $p$ imply that
all the coefficients of $r(T)$ are multiples of $p$. Therefore,
\begin{align*}
(\nu_1(T),\, f(T)) &= (\nu_1(T),\, r(T))\\
&\subseteq (\nu_1(T), \, p),
\end{align*}
which has index $p^{p-1}$ in $\Lambda$. 

Finally, suppose that $E$ contains a direct sum
$$
\bigoplus_i \Lambda/(f_i(T))
$$
with $\deg f_i(T) = \lambda_i $ and $\sum_i \, \lambda_i = \lambda$.
Since
$$
\sum_i \text{min}(\lambda_i, p-1) \ge \text{min}(\lambda, p-1),
$$
and the $i$th summand contributes at least $\text{min}(\lambda_i, p-1)$ to $E$, the theorem follows, except
for the extra claim for $p=2$.

When $p=2$, we have $\nu_2(T)=T^3+4T^2+6T+4$. Suppose $\lambda = 2$. A summand $E_1=\Lambda/(T-a)$
contributes at least $2^{2}$ to the order of $E/\nu_2E$. Therefore, if there are two or more such
summands in $E$, we have $(E/\nu_2E)\ge 2^4$. 

The analysis of a summand $\lambda/(f(T))$ with $f(T)=T^2+aT+b$ is more technical. If we divide 
$f(T)$ into $\nu_2(T)$, we obtain a remainder of the form $xT+y$, so
$$
(\nu_2(T), \, f(T)) = (f(T),\,  xT+y).
$$
\begin{lemma} Let $f(T)=T^2+aT+b$ with $2\mid a$ and $2\mid b$.
The index of the ideal $(f(T),\,  xT+y)$ in $\Lambda$ is $2^v$, where $v$ is the 2-adic valuation of
$bx^2-axy+y^2$.
\end{lemma}
\begin{proof} $\Lambda/(f(T))$ is a free $\mathbb Z_2$-module with basis $\{1, T\}$. Multiplication by
$xT+y$ mod $f(T)$ maps this module to itself. Its matrix with respect to this basis is
$$
\begin{pmatrix} y & -bx \\ x & y-ax\end{pmatrix}.
$$
The power of 2 in the determinant of this matrix gives the order of the cokernel of the map, which is the desired index.
\end{proof}
In our case, the determinant is
$$
-4a^3+a^2(6b+16)-a(4b^2+12b+24) + b(b+2)^2.
$$
Each term, except possibly $b(b+2)^2$, is clearly divisible by 16.
Consideration of the cases $b\equiv 2\pmod 4$ and $b\equiv 0 \pmod 4$ shows that $b(b+2)^2$ is also
a multiple of 16, so the determinant is a multiple of $16$. The lemma yields the theorem in this case.

It remains to treat the case where $p=2$ and  $\deg f(T) = \lambda\ge 3$. 
We have
\begin{align*}
(\nu_2(T),\, f(T)) & \subseteq (\nu_2(T),\, 2,\, f(T))\\
&= (\nu_2(T),\, 2, \, f(T)-T^{\lambda-3} \nu_2(T))\\
&= (\nu_2(T),\, 2),
\end{align*}
which has index $2^{3}$ in $\Lambda$.

This completes the proof for $p=2$ when $E$ is of the form $\Lambda/(f(T))$. When $E$ contains a sum of
such modules, the contributions from each summand are added together, so
the result follows immediately from the cases considered. This completes the proof of the theorem. \end{proof}

 {\bf Remarks.} The results in the theorem are sharp, as the following examples show:
\begin{enumerate}
\item Let $N\ge 0$ and let $E=\Lambda/(f(T))$, where $f(T)=T^N \nu_1(T) + p$. Then 
$\lambda = p-1+N$ and
$$
 \#(E/\nu_1E) = p^{p-1}.
$$
\item Let $p=2$, let $f(T) = \nu_2(T) + 2$, and let $E=\Lambda/(f(T))$. Then 
$$
(\nu_2(T), \, f(T)) = (2, \, f(T)), $$
which has index 8 in $\Lambda$. 
\end{enumerate}

\section{Pseudo-isomorphisms}\label{Pseudo}

In this section, we show how knowledge of $A_0$ and $A_1$ can be used to evaluate the Iwasawa invariants
in some cases. Our method was inspired by the work of Gold \cite{Gold}, who used such information
to evaluate $\lambda$ for cyclotomic $\mathbb Z_p$-extensions.

We assume that the primes above $p$ are totally ramified in $K_{\infty}/K_0$.
By \cite[Lemma 13.15]{W}, there is a submodule $Y_0$ of $X$ such that
$$
X/\nu_n Y_0 \simeq A_n
$$
for each $n\ge 0$.

The structure theorem for $\Lambda$-modules says that there is an elementary $\Lambda$-module
$$
E=\bigoplus_i \Lambda/(f_i(T)),
$$
where each $f_i$ is either a power of $p$ or a distinguished polynomial, and there is an exact sequence
$$
\begin{CD}
0@>>> F_1  @>>> Y_0 @>>> E @>>> F_2 @>>>0
\end{CD}
$$
of $\Lambda$-modules with $F_1$ and $F_2$ finite.

\begin{theorem}\label{PropOrd}
Suppose $E$ is an elementary $\Lambda$-module and there is an exact sequence
$$
\begin{CD}
0@>>> F_1  @>>> Z @>>> E @>>> F_2 @>>>0
\end{CD}
$$
of $\Lambda$-modules, where $F_1$ and $F_2$ are finite. Let $f\in \Lambda$ be such that $Z/fZ$ is finite. Then
$$
\#(Z/fZ) = \#(F_1/fF_1)\, \#(E/fE).
$$
\end{theorem}
\begin{proof} Let $Y=Z/F_1$.
There is a commutative diagram
$$
\begin{CD}
0@>>> Y @>>> E @>>> F_2 @>>> 0\\
 @.@VVV @VVV@VVV @.\\
0@>>> Y @>>> E @>>> F_2 @>>> 0\\
\end{CD}
$$
where the vertical maps are multiplication by $f$.

Suppose that multiplication by $f$ is injective on $E$. The Snake Lemma says there is an exact sequence
$$
\begin{CD}
0@>>> F_2[f] @>>> Y/f Y @>>> E/f E @>>> F_2/f F_2 @>>>0,
\end{CD}
$$
where $F_2[f]$ is the kernel of multiplication by $f$.
The exact sequence 
$$
\begin{CD}
0@>>> F_2[f]@>>> F_2@>>>F_2@>>>F_2/fF_2 @>>>0
\end{CD}
$$
shows that $\#F_2[f] = \#(F_2/fF_2)$. Therefore,
$$
\#(Y/f Y) = \#(E/f E).
$$

If multiplication by $f$ is not injective on $E$, then $E/f E$ is infinite.
The exact sequence
$$
\begin{CD}
 F_2[f] @>>> Y/f Y @>>> E/f E @>>> F_2/f F_2 @>>>0
\end{CD}
$$
shows that $Y/f Y$ is also infinite. 

Therefore, $\#(Y/fY) = \#(E/fE)$ in both cases.

Consider the diagram
$$
\begin{CD}
0@>>> F_1 @>>> Z @>>> Y @>>> 0\\
 @.@VVV @VVV@VVV @.\\
0@>>> F_1 @>>> Z @>>> Y @>>> 0\\
\end{CD}
$$
where the vertical maps are multiplication by $f$.
Because $Z/fZ$ surjects onto $Y/fY$, the module $Y/fY$
must be finite. Therefore, $E/fE$ is finite, which implies that
$f$ and the characteristic power series of $E$ are relatively prime. 
Because $Y$ injects into $E$, multiplication by $f$ is injective 
on $Y$. The Snake Lemma yields the exact sequence
$$
\begin{CD}
0@>>> F_1/fF_1  @>>> Z/fZ @>>> Y/fY @>>>0.
\end{CD}
$$
Therefore,
$$
\#(Z/fZ) = \#(F_1/fF_1) \#(Y/fY)= \#(F_1/fF_1) \#(E/fE).
$$
This proves the theorem. \end{proof}

{\bf Remark.} The example $Z=(p, T)/(p^2)$, $E=\Lambda/(p^2)$, $F_1=0$, $f=T$ yields
$Z/TZ\simeq p\times p$ and $E/TE\simeq p^2$, so the groups must have the same order but are not necessarily isomorphic.

We apply Theorem \ref{PropOrd} where $Z=Y_0$ and $E$ is an elementary $\Lambda$-module.
Since $Y_0/\nu_nY_0$ is a submodule of $X/\nu_nY_0$, which is isomorphic to $A_n$, it follows that
$Y_0/\nu_nY_0$ is finite. 

Suppose that $\mu>0$, so $E$ contains a summand of the form $\Lambda/(p^j)$.
Such a term contributes $p^{j(p^n-1)}$ to the order of $E/\nu_nE$, and $\mu$ is the sum of the values of $j$
for the submodules of this form, so we find that
$$
e_n-e_0\ge (p^n-1)\mu.
$$

We have proved the following.
\begin{proposition}\label{usefulprop}
$$
\#A_n = \#(X/Y_0) \#(Y_0/\nu_n Y_0) = \#A_0 \#(F_1/\nu_n F_1) \#(E/\nu_n E).
$$
Therefore,
$$
 p^{e_n-e_0}=\#(F_1/\nu_n F_1) \#(E/\nu_n E)\ge p^{\mu(p^n-1)}.
$$
In particular, 
$$
\mu\le \frac{e_n-e_0}{p^n-1}.
$$
\end{proposition}

\section{Totally Ramified}\label{SectTotRam}

In order to apply our results, we need that $K_{\infty}/K_0$ is totally ramified
at the primes above $p$. The following two results treat the situations that we need.

For $p=3$, we use the following, which is proved for example in \cite{Chang}.

\begin{proposition} Let $d=hk^2$ with squarefree integers $h, k$ such that $\gcd(h,k)=1$, $h\ne 1$, and $3\nmid k$. Then the discriminant
of $\mathbb Q(\zeta_3,\sqrt[3]{d})/\mathbb Q$ is
$$
-3^7(d/k)^4 \text{ if } d\not\equiv \pm 1\mod 9, \quad -3^3(d/k)^4 \text{ if } d\equiv \pm 1\pmod 9
$$
\end{proposition}
\begin{proof} It was proved by Dedekind (\cite[pp. 53, 54]{Dedekind}; see also \cite[pp. 38, 49-51]{Marcus})
that the discriminant of $\mathbb Q(\sqrt[3]{d})$ is $-27(hk)^2$ if $d\not\equiv \pm 1\pmod 9$, and $-3(hk)^2$
if $d\equiv \pm 1\pmod 9$. Since the extension $\mathbb Q(\sqrt[3]{d}, \zeta_3)/\mathbb Q(\sqrt[3]{d})$ is tamely ramified at
the prime above 3 and is unramified at all other primes, the relative discriminant is the prime of $\mathbb Q(\sqrt[3]{d})$ above 3 to 
raised to the first power. The formula for discriminants in towers yields the desired result. \end{proof}

It follows that the norm ${N\mathcal D}$ of the relative discriminant
of $\mathbb Q(\zeta_3,\sqrt[3]{d})/\mathbb Q(\zeta_3)$ satisfies
$$
v_3(N{\mathcal D}) = \begin{cases} 8 \text{ if } 3\mid d\\  4\text{ if } 3\nmid d, d\not\equiv \pm 1\pmod 9\\
0 \text{ if } d\equiv \pm 1\pmod 9,\end{cases}
$$
where $v_3$ is the 3-adic valuation, normalized by $v_3(3)=1$.
The conductor-discriminant formula says that ${\mathcal D}$ is the product of the conductors of the characters
for $\text{Gal}(\mathbb Q(\zeta_3,\sqrt[3]{d})/\mathbb Q(\zeta_3))$. 

Now consider the extension $K_1/k_0=\mathbb Q(\zeta_3, \sqrt[3]{3}, \sqrt[3]{d})/\mathbb Q(\zeta_3)$. The conductor-discriminant formula
tells us that the relative discriminant is the product of the conductors of the characters of the Galois group, and this is the product of
the discriminants of the four cubic subextensions $L_i/\mathbb Q(\zeta_3)$. The formula for discriminants in towers implies that
the discriminant of $\mathbb Q(\zeta_3, \sqrt[3]{3},\sqrt[3]{d})/\mathbb Q$ is, up to sign,
$3^9$ times the product of the norms of the relative discriminants of these four cubic extensions $L_i/\mathbb Q(\zeta_3)$.

Putting everything together, 
we find that the 3-adic valuation of the discriminant $D$ of  $\mathbb Q(\zeta_3, \sqrt[3]{3},\sqrt[3]{d})/\mathbb Q$ is
$$
v_3(D) = \begin{cases} 37 \text{ if } 3\nmid d, d\not\equiv \pm 1\pmod 9\\
33 \text{ if } d\equiv \pm 1\pmod 9\\
37 \text{ if } 3\mid d, d/3\not\equiv \pm 1\pmod 9\\
33 \text{ if } 3\mid d,  d/3\equiv \pm 1\pmod 9.
\end{cases}
$$

The 3-adic valuation of the discriminant of $\mathbb Q(\zeta_3, \sqrt[3]{d})/\mathbb Q$ is either 3, 7, or 11,
with the 11 occurring when $3\mid d$. This is less than 1/3 of the 3-adic valuation of the discriminant of
$\mathbb Q(\zeta_3,\sqrt[3]{3}, \sqrt[3]{d})/\mathbb Q$ except when $d=3d_1$ with $d_1\equiv \pm 1\pmod 9$.
We have proved the following:
\begin{proposition}
Let $d\ne 1, 3$ be a cubefree integer with $9\nmid d$. The extension 
$$
\mathbb Q(\zeta_3, \sqrt[3]{3}, \sqrt[3]{d})/\mathbb Q(\zeta_3, \sqrt[3]{d})
$$
is totally ramified at the primes above 3 except when $d=3d_1$ with $d_1\equiv \pm 1\pmod 9$, in which case it is
unramified.
\end{proposition}

For $p=2$, we use the following.
\begin{proposition}\label{totram2}
Let $d$ be odd and squarefree. Let $K_0=\mathbb Q(\sqrt{-1}, \sqrt{d})$ and $K_1=\mathbb Q(\sqrt{-1}, \sqrt{d}, \sqrt{2})$.
Then $K_1/K_0$ is totally ramified at the primes of $K_0$ above 2.
If $d$ is odd and squarefree and $d\not
\equiv \pm 1 \pmod 8$, then there is only one prime of $K_0$ dividing $2$.
\end{proposition}
\begin{proof}  An easy calculation with the conductor-discriminant formula shows that
the discriminant $D_0$ of $K_0$ is $-16d^2$ and the discriminant $D_1$ of $K_1$ is $2^{16}d^4$.
Since $D_1/D_0^2= 2^8$ is even, the primes above 2 ramify in $K_1/K_0$.

Since $2$ is inert in $\mathbb Q(\sqrt{\pm d})/\mathbb Q$, where $\pm d\equiv 1\pmod 4$,
and 2 ramifies in $\mathbb Q(\sqrt{-1})/\mathbb Q$,
there is only one prime of $K_0$ over 2.
\end{proof}

For a prime of $K_0$, the inertia subgroup of $\text{Gal}(K_{\infty}/K_0)$ surjects onto the inertia subgroup of $K_1/K_0$.
If $K_1/K_0$ is ramified then $K_{\infty}/K_0$ is totally ramified. This is what we need for our examples in Section \ref{SectComp}.

The following is useful in calculations using Proposition \ref{Chev}.
\begin{proposition}\label{unram2} In the notation of Proposition \ref{totram2}, $K_n/k_n$ is unramified at the primes above 2 for all $n\ge 0$.
\end{proposition}
\begin{proof} Since $d$ is odd, either $d\equiv 1\pmod 4$ or $-d\equiv 1\pmod 4$,
so $K_0=\mathbb Q(\sqrt{-1}, \sqrt{d})=\mathbb Q(\sqrt{-1}, \sqrt{-d})$ is an unramified extension of $k_0$, \end{proof} 

\section{Conjectural calculation of $A_2$}\label{conj}

When $p=3$, the field $K_2$ has degree 54 over $\mathbb Q$, which makes computation of its class number
difficult. In this section, we suggest a conjectural method for computing, or at least estimating, its class number
in many cases. (Actually, all of the calculations in this paper are conjectural, since Pari-GP assumes the Generalized Riemann
Hypothesis in its algorithms, but the methods of this section are probably much less trustworthy then GRH.)
The present algorithm was inspired by a method used by Van Huele \cite{VH}, where fixed fields of a Klein 4-group and its subgroups
were used to calculated class numbers. In the present case, we are using the fixed fields of an $S_3$ and its subgroups.

Let $K_n'$ be any subfield of $K_n$ such that $[K_n : K_n']=2$. All such fields are conjugate over $\mathbb Q$, so the choice does not matter.
Let $A_n'$ and $h_n'$ be the $3$-Sylow subgroup of the class group and the class number of $K_n'$.
Let $h_n$ be the class number of $K_n$. Honda \cite{Honda} showed that $h_0 = (h_0')^2$ or $(1/3)(h_0')^2$. Moreover, Schoof \cite{Schoof}
showed that if $d$ contains no prime factors $q\equiv 1\pmod 3$ then $h_0 = (h_0')^2$. It seems likely that a result similar to this
should hold for $h_n$ and $h_n'$. 

In the numerical examples in the next section, we compute the Iwasawa invariants for various $d$. In all of the cases where
we are able to evaluate $\lambda, \mu, \nu$ exactly, we can deduce the actual value of the 3-part of $h_2$.
We also calculated $h_2'$, which is much easier than $h_2$ because $[K_2' : \mathbb Q]=27$. This suggested the following.

\begin{conjecture} Suppose $d$ is a product of primes $q\equiv 2$ or $5\pmod 9$ and $d\not\equiv \pm 1\pmod 9$.
Then $h_2= (h_2')^2$.
\end{conjecture}

The conjecture holds for all of the numerical examples of the next section for which $d$ satisfies the specified conditions
and for which we obtained exact values for $\lambda, \mu, \nu$ (except possibly $d=230$, where the calculation of $h_2'$ took too long to complete).
In these examples, we also have $h_1=(h_1')^2$. 

Some condition must be made on the primes dividing $d$. 
For $d=1870$, we know that $\mu=3, \lambda=0, \nu\ge 4$, and, in fact, 
the computations that give these values yield a lower bound on the power of 3 dividing $h_2$, so we have  $3^{31}\mid h_2$.
However, $3^{15}\mid\mid h_2'$, so $h_2\ge 3(h_2')^2$ (we are ignoring the non-3-parts of the class numbers, but it can be shown
that these agree for $h_n$ and $(h_n')^2$). For $d=1870$, we also have $A_1=9\times 3^{11}$ and $A_1'=3^6$, so
$h_1=3(h_1')^2$.

One might be tempted to guess that a stronger statement is true: $A_n\simeq A_n'\times A_n'$. However,
for $d=5\cdot 11\cdot 173$, we have $A_1=27\times 9^2\times 3^7$, but $A_1'=9^2\times 3^3$. Therefore, $h_1=(h_1')^2$
but $A_1$ is not isomorphic to $A_1'\times A_1'$.

For $d= 10$, $44$, $46$, $253$, $5\cdot 29$, $17\cdot 53$, $5\cdot 11\cdot 53$, and $5\cdot 17\cdot 53$, 
we have $h_1=(h_1')^2$, but for the following values of $d$ we have $h_1=(1/3)(h_1')^2$:
$17,  53,  55,  2\cdot 167, 2\cdot 257, 5\cdot 83, 5\cdot 11\cdot 17$. The cases in this latter list with $d\equiv 1\pmod 9$ are why we exclude $d\equiv \pm 1\pmod 9$ from the conjecture.

It seems likely that $h_n/(h_n')^2$ is bounded above and below by constants depending on $n$. Honda's result
gives the $n=0$ case, with a lower bound of 1 and an upper bound of 3.

\section{Numerical examples}\label{SectComp}

In the notation of Section 6, we have an exact sequence
$$
0\to F_1\to Y_0 \to E \to F_2\to 0,
$$
where $E$ is elementary and $F_1, F_2$ are finite.
The general strategy is to write 
$$E =E_1\oplus E_2,
$$
where $E_1$ is a sum of terms of the form $\Lambda/(p^j)$ and $E_2$ is what remains,
and then analyze the orders of $E_2/\nu_nE_2$ and $F_1/\nu_nF_1$. 
Proposition \ref{usefulprop} yields the following.
\begin{proposition}\label{ene0}
$$
p^{e_n-e_0-\mu(p^n-1)} = \#(F_1/\nu_n F_1)\, \#(E_2/\nu_n E_2).
$$
\end{proposition}

The following is useful for identifying $F_1$ and $E_2$.
\begin{lemma}\label{LemmaNonDec} Let $M$ be a finitely generated $\Lambda$-module. Then
$ M/\nu_1 M = 0$ if and only if $M=0$. More generally,
$\#(M/\nu_n M) \le \#(M/\nu_{n+1}M)$ for all $n$, and 
if $\#(M/\nu_n M) = \#(M/\nu_{n+1}M)<\infty$  for some $n$, then $\nu_nM=0$.
\end{lemma}
\begin{proof}  Since $\nu_0=1$, the first statement is a special case of the second. Since $\nu_n$ divides $\nu_{n+1}$, 
there is a natural surjection $M/\nu_{n+1}M\to M/\nu_n M$.
If the orders are equal and finite, the map is an isomorphism, so $\nu_{n}M=\nu_{n+1}M= (\nu_{n+1}/\nu_n)\nu_nM$. 
 Since $\nu_{n+1}/\nu_n$ is in the maximal ideal of $\Lambda$, Nakayama's Lemma implies that $\nu_nM=0$. 
\end{proof}

\subsection{$\mathbf{p=3}$}

Using Pari-GP \cite{Pari}, we computed several examples. Let 
$$K_0=\mathbb Q(\zeta_3, \sqrt[3]{d}).
$$
We concentrate on the case
where $d$ is a product of primes that are 2 mod 3.
\medskip

\noindent
{$\mathbf{d= 22, 34, 58, 68, 85, 92, 164, 236:}$}
In these cases, the 3-parts of the class groups are $A_0=3^2$ and $A_1=3^4$. 
Therefore, we have $e_1-e_0=2$.

 Since $e_1-e_0\ge (3-1)\mu$, and $\mu>0$, we must have $\mu=1$
and $E$ contains one summand $E_1=\Lambda/(3)$. Write $E=E_1\oplus E_2$.
Then $\#(E_1/\nu_1E_1)=3^2$.
Proposition \ref{ene0} tells us that
\begin{equation*}\label{d22F}
 1=\#(F_1/\nu_1F_1) \#(E_2/\nu_1E_2).
\end{equation*}
Lemma \ref{LemmaNonDec} implies that $E_2=F_1=0$.
Therefore, 
$$
\#A_n = \#A_0\, \#(\Lambda/(3,\nu_n)) = 3^{3^n + 1}.
$$
This says that
$$
\mu=1, \quad \lambda = 0, \quad \nu=1.
$$
\medskip

\noindent
$\mathbf{d=10, 44, 46, 253:}$ In these cases, $e_0=0$ and $e_1=2$. This yields, as in the previous example, 
$$
\mu=1, \quad \lambda=0, \quad \nu=-1.
$$

\medskip
\noindent
$\mathbf{d=110, 230:}$
Since three primes that are 2 mod 3 divide $d$, we have $\mu \ge 2$. Calculations yield $A_0=3^4$ and $A_1=3^8$, so $e_1-e_0 = 4$. This implies that
$\mu\le 2$, so $\mu = 2$. Therefore, $\#(E_1/\nu_1E_1)=3^4$. Proposition \ref{ene0} implies that
$$
1=\#(F_1/\nu_1F_1)\#(E_2/\nu_1E_2),
$$
so $E_2=F_1=0$. 
It follows that
$$
\mu=2, \quad  \lambda = 0, \quad \nu=2.
$$
\medskip

\noindent
$\mathbf{d=170:}$ We have $A_0= 3^2$ and $A_1=3^6$, so $e_1-e_0=4$. As in the previous example, 
$$
\mu=2, \quad  \lambda = 0, \quad \nu=0.
$$

\medskip

\noindent
{$\mathbf{d=1870:}$} 
We know $\mu \ge 3$. We have $3^6$ at the 0-th level and $3^{11} \times 9$ at the first level, so $e_1-e_0=7$.
This yields $\mu\le 3$, so $\mu=3$. 

We now have
$$
3=\#(E_2/\nu_1E_2) \#(F_1/\nu_1F_1).
$$
If $\lambda >0$, we must have $\lambda\ge 2$ by Theorem \ref{PropEven} or Theorem \ref{LambdaEvenMore}. From Theorem \ref{LambdaThm},
we deduce that
$$
\#(E_2/\nu_1E_2) \ge 3^2,
$$
which in impossible. Therefore, $\lambda = 0$ and $E_2=0$ and
we must have $\#(F_1/\nu_1F_1)= 3$, hence $\#(F_1/\nu_nF_1)\ge 3$ for all $n\ge 1$.
Therefore,
$$
\#A_n=\#A_0 \#(F_1/\nu_n F_1) 3^{3 (3^n-1)}= \#(F_1/\nu_n F_1) 3^{3\cdot 3^n +3}\ge 3^{3\cdot 3^n +4}.
$$
In summary,
$$
\mu=3, \quad \lambda  = 0, \quad \nu\ge 4.
$$

\medskip
\noindent
$\mathbf{d=94, 115, 205, 295:}$ 
We have $A_0=3^2$ and $A_1=3^6$, so $e_1-e_0 = 4$. This means that $\mu = 1$ or $2$, 

Let's assume the conjecture of Section \ref{conj}. In the notation of that section, $3^6\mid\mid h_2'$ for
$d=94$ and $295$. The conjecture implies that $e_2=12$, so $e_2-e_0=10 \le (3^2-1)\mu$. Therefore, $\mu=1$.
Therefore,
$$
3^{6} = \#A_1=\#A_0 \#(F_1/\nu_1 F_1) \#(E_2/\nu_1E_2) 3^{3^1-1}=3^{4} \#(F_1/\nu_1 F_1) \#(E_2/\nu_1 E_2)
$$
and 
$$
3^{12} = \#A_2=\#A_0 \#(F_1/\nu_2 F_1) \#(E_2/\nu_2E_2) 3^{3^2-1}=3^{10} \#(F_1/\nu_2 F_1) \#(E_2/\nu_2 E_2).
$$
It follows that $\nu_1E_2=0$ and $\#F_1=3$, which yields $\lambda=0$ and $\nu=3$.

Note that if the conjecture fails for these values of $d$, but is only off by a small power of 3, then we still obtain $\mu=1$ but we do not obtain values
for $\lambda$ and $\nu$.

For $d=115$, we have $3^7\mid\mid h_2'$. The conjecture yields $14-2=e_2-e_0\ge (3^2-1)\mu$, so $\mu=1$. We do not obtain the values of $\lambda$ and $\nu$.

For $d=205$, we have $3^8\mid\mid h_2'$. Since $d\equiv 1\pmod 9$, the conjecture does not apply. However, if $h_2\mid 3(h_2')^2$, we have $17-2\ge (3^2-1)\mu$, hence $\mu=1$.

\medskip
\noindent
$\mathbf{d=5\cdot 11\cdot 173:}$ We must have $\mu\ge 2$. Computation yields
$A_0= 9^2 \times 3^2$  and  $A_1= 27 \times 9^2 \times 3^7$.
So $e_1-e_0=8$, hence $\mu\le 4$. 
If $\mu=4$ then $F_1=E_2=0$, so $\lambda=0$ and $\nu=0$. Proposition \ref{usefulprop}
tells us that $e_n=4\cdot 3^n +2$. However, if
$\mu=2$ or $3$, we do not obtain information on $\lambda$ and $\nu$.

\medskip
\noindent
$\mathbf{d=17\cdot 53\cdot 71:}$ We have $A_0 = 3^4$ and $A_1= 81\times 27\times 3^8$. Therefore,
$e_1-e_0=11$, so $2\le \mu \le 5$. The structure of the groups suggests $\lambda=2$.

\medskip

\subsection{$\mathbf{p=2}$}

The advantage of $p=2$ is that the fields $K_n$ have lower degrees over $\mathbb Q$ than for other choices of $p$.
This allows computations of class groups for more values of $n$.

Let $d$ be odd and squarefree and let 
$$k_0=\mathbb Q(i),\qquad K_0=k_0(\sqrt{d}).$$

\medskip

\noindent
$\mathbf{d=21:}$ The class groups of $K_0, K_1, K_2$ are $2, 2^2, 2^4$. Therefore, $e_1-e_0=1$, which implies 
that $\mu\le 1$, hence $\mu=1$. Therefore
$$
1=2^{e_1-e_0-\mu(2-1)}=\#(F_1/\nu_1F_1)\#(E_2/\nu_1E_2).
$$
Lemma \ref{LemmaNonDec} implies that $F_1=E_2=0$, which yields  
$$
\mu=1, \quad \lambda=0, \quad \nu=0.
$$
Proposition \ref{usefulprop} tells us that $e_n=2^n$ for all $n$.
\medskip

\noindent
$\mathbf{d=33, 57:}$ The class groups of $K_0, K_1, K_2, K_3$ are $2, 4\times 2, 8\times 2^3, 16\times 2^7$, so
$e_1-e_0=2$, $e_2-e_0=5$, $e_3-e_0=10$. Therefore, 
$$\mu=1.$$
Proposition \ref{ene0} implies that
\begin{equation}\label{ene02}
2^{e_n-e_0-(2^n-1)} = \#(F_1/\nu_n F_1) \#(E_2/\nu_n E_2)
\end{equation}
for all $n\ge 0$. For $n=1$, the left side of this equation is $2$,  for $n=2$ the left side is $2^2$, and for $n=3$ the left side is $2^3$. 
Theorem \ref{LambdaThm} implies that $\lambda\le 1$. The existence of cyclic subgroups $2, 4, 8, 16$ at levels 0, 1, 2, 3
makes it likely that $\lambda=1$, which Theorem \ref{PropEven} does not prohibit.  If $E_2$ has $\lambda=1$, then $\#(E_2/\nu_1E_2)\ge 2$
and $\#(E_2/\nu_2E_2)\ge 4$. Equation (\ref{ene02}) implies that these are equalities, so $F_1=0$, but this does not
yield the value of $\nu$. However, we have $e_n=2^n+n$ for $n=0, 1, 2, 3$, so it seems likely that $\mu=1$, $\lambda=1$, and $\nu=0$.
\medskip

\noindent
{$\mathbf{d=3\cdot 7\cdot 11:}$} The class groups of $K_0, K_1, K_2, K_3$ are
$$
12\times 2, \quad 12\times 2^3, \quad 12\times 2^7, \quad 12 \times 2^{15},
$$
respectively. We know that $\mu\ge 2$. Since 
$e_1-e_0=2$, we have $\mu=2$. 
Proposition \ref{ene0} yields that
$$
1=2^{e_1-e_0-\mu\cdot (2-1)} = \#(F_1/\nu_1 F_1)\#(E_2/\nu_1 E_2).
$$
Therefore, $F_1=E_2=0$, from which it follows that
$$
\mu=2, \quad \lambda = 0, \quad \nu=1.
$$

\medskip
\noindent
{$\mathbf{d=3\cdot 11\cdot 23:}$} The class groups of $K_0, K_1, K_2$ are
$24\times 2$, $24\times 6\times 2^2$, and  $24\times 6\times 2^6$, so $e_1-e_0=2$, $e_2-e_0 = 6$.
This yields $\mu=2$. Therefore,
$$
1 = 2^{e_1-e_0-\mu(2-1)}  \#(F_1/\nu_1 F_1)\#(E_2/\nu_1 E_2).
$$
Therefore,
$$
\mu=2, \quad \lambda = 0, \quad \nu = 2.
$$

\medskip
\noindent
{$\mathbf{d=3\cdot 19\cdot 23:}$} The class groups of $K_0, K_1, K_2$ are
$28\times 2$, $84\times 2^3$, and  $84\times 2^7$, so $e_1-e_0 = 2$, $e_2-e_0=6$.
This yields
$$
\mu=2,\quad \lambda = 0,\quad \nu = 1.
$$

\medskip
\noindent
{$\mathbf{d=7\cdot 11 \cdot 19:}$} The class groups of $K_0, K_1, K_2$ are
$ 32 \times 2$, $160\times 2^3$,  and $160 \times 10^2 \times 2^5$, 
so $e_1-e_0= 2$,  $e_2-e_0 = 6$. This yields
$$
\mu = 2,\quad \lambda = 0,\quad \nu = 4.
$$

\medskip
\noindent
{$\mathbf{d=3\cdot 11\cdot 19:}$} The class groups of $K_0, K_1, K_2, K_3$ are
$4^2$, $12\times 4^3$, $48\times 16\times 8\times 4\times 2^4$, and
$96\times 16 \times 8\times 4 \times 2^{12}$.
We know $\mu \ge 2$.
Since $e_3-e_0= 22$, we have $\mu \le 3$.

If $\mu=3$, then
$$
2^{e_n-e_0-3(2^n-1)} = \#(F_1/\nu_n F_1)\#(E_2/\nu_n E_2)
$$
for all $n$. This yields 
$$
\#(F_1/\nu_2F_1)\#(E_2/\nu_2 E_2) = 3 < 3^4 = \#(F_1/\nu_1F_1)\#(E_2/\nu_1 E_2),
$$
which is impossible by Lemma \ref{LemmaNonDec}. Therefore,
$$\mu=2.$$
Our techniques cannot determine $\lambda$ and $\nu$ in this case, but one guess is
$\lambda >0$. Since a positive $\lambda$ corresponds to $A_n$ having unbounded exponent,
having $\lambda>0$ is consistent with the occurrence of a cyclic subgroup of order 16 in $A_2$ and the
cyclic subgroup of order 32 in $A_3$. If this is the case, then Theorem \ref{PropEven} says that $\lambda\ge 2$.

\medskip
\noindent
{$\mathbf{d=11\cdot 19\cdot 23:}$} The class groups here are large enough that it is hard to determine the Iwasawa invariants.
The class groups of $K_0, K_1, K_2$ are
$80\times 4$, $320\times 16\times 4\times 2$, and $1920\times 96\times 8^2\times 4\times 2^3$, respectively, so
$e_1-e_0= 7,  e_2-e_0 = 17$. This yields $2\le  \mu \le 5$.
The occurrence of subgroups $16\times 4$, $64\times 16$, and $128\times 32$ indicates a possibility of $\lambda \ge 2$
(however, Theorem \ref{PropEven} does not apply in this case, so we do not know the parity of $\lambda$).
Of course, computation of $A_3$ should yield some insight, but this computation exceeded the capacity of
our resources.

\medskip
\noindent
{$\mathbf{d=3\cdot 7\cdot 19:}$} The class groups of $K_0, K_1, K_2, K_3$ are
$16\times 4$, $16\times 4^2\times 2$, $16\times 4^4\times 2^3$, and $16\times 8^2\times 4^4\times 2^9$, so
$e_0=6, e_1-e_0=3, e_2-e_0=9, e_3-e_0=21$. This yields $2\le \mu\le 3$. There is not a strong indication of positive $\lambda$,
so possibly $\mu=3$, $\lambda=0$, $\nu=3$, which fits the data well.

\medskip
\noindent
{$\mathbf{d=3\cdot 7\cdot 11 \cdot 19\cdot 23:}$} The class groups of $K_0, K_1, K_2$ are
$16\times 4^2\times 2$, $480\times 8\times 4\times 2^5$, and $960\times 16\times 4\times 2^{13}$,
so      $e_1-e_0 = 6,   e_2-e_0 = 16$.
This yields $4\le \mu \le 5$. The occurrence of subgroups $16\times 4$, $32\times 8$, $64\times 16$
is consistent with $\lambda \ge 2$. 

Suppose $\mu=5$. Then 
$$
2^{e_n-e_0-5(2^n-1)} = 2
$$
for $n=1$ and $n=2$,
which would yield $\lambda=0, \nu = 5$. Therefore, we suspect that $\mu=4$, $\lambda = 2$, and $\nu=5$. 

\medskip
\noindent
{$\mathbf{d=3\cdot 7\cdot 11 \cdot 19\cdot 31:}$} The class groups of $K_0, K_1, K_2$ are
$32\times 4^2 \times 2$, $576\times 8\times 4\times 2^5$, and  $8064\times 336\times 12 \times 2^{13}$, so  
$e_1-e_0 = 6,   e_2-e_0 = 16$. This yields $4\le \mu \le 5$. As in the previous example,
we suspect that $\mu=4$, $\lambda = 2$, and $\nu=6$.

\section{Ambiguous Classes}\label{SectAmbig}

A natural question is whether all of $\mu$ can be explained by Chevalley's formula. In this section, we examine the
ambiguous classes that contribute to this formula with the eventual hope of understanding the full cause of positive $\mu$ invariants.

Let $K/k$ be a cyclic extension of prime degree $p$, where $k$ is a number field with class number prime to $p$,
and let $\tau$ generate $\text{Gal}(K/k)$. Let $A$ be the $p$-Sylow subgroup of the ideal class group of $K$
and let
$$
A^{\tau} = A[1-\tau]
$$
be the {\it ambiguous subgroup} of $A$, namely, the ideal classes fixed by $\tau$. Define a {\it strongly 
ambiguous} ideal class to be an ambiguous ideal class that contains an ideal fixed by $\tau$.
Since the class number of $k$ is assumed to be prime to $p$, the strongly ambiguous classes are generated
by the ramified primes.

Since $k$ has class number prime to $p$, it follows that $N=1+\tau+\tau^2+\cdots +\tau^{p-1}$ annihilates $A$.
Since $N$ acts as $p$ on $A^{\tau}$, we see that $A^{\tau}$ has exponent dividing $p$.

\begin{proposition}\label{ambig} The group of ambiguous ideal classes modulo the strongly ambiguous ideal classes
is isomorphic to $E_k \cap N(K^{\times})/N(E_K)$, where $E_K$ and $E_k$ are the unit groups for $K$ and $k$.
\end{proposition}
\begin{proof}  Start with the two exact sequences
\begin{gather*}
1\to E_K\to K^{\times} \to P_K\to 1\\
1\to P_K \to I_K \to C_K\to 1,
\end{gather*}
where $P_K$ denotes the group of principal fractional ideals of $K$, $I_K$ is the group of fractional ideals
of $K$, and $C_K$ is the ideal class group of $K$. 
Since Tate cohomology for the cyclic group $\text{Gal}(K/k)$ is periodic with period 2,
we have the exact sequences
\begin{gather*}
I_K^{\tau} \to C_K^{\tau} \to H^1(P_K) \to 0\\
0\to H^1(P_K) \to E_k/N(E_K) \to k^{\times}/N(K^{\times}),
\end{gather*}
where the cohomology groups are for the group $\langle \tau\rangle$. 
The proposition follows easily. \end{proof}

\begin{corollary}\label{strongambig} The number of strongly ambiguous classes is
$$
\frac{\prod_{\mathcal P \in S} e_P}{p [E_k  : N(E_K)]}.
$$
\end{corollary}
\begin{proof}
Combine Proposition 1 and Proposition \ref{ambig}. Since our ambiguous classes
are taken only from the $p$-Sylow subgroup, and $k$ has class number prime to $p$, the factor
$h$ is omitted. \end{proof}

Fix a generator $\tau$ of $\text{Gal}(K_\infty/k_\infty)$. We can regard $\tau$ as a generator
of $\text{Gal}(K_n/k_n)$ for each $n$. Then $\tau$ acts on $X$, and this action commutes with the action of $\Lambda$.
Therefore, $X^{\tau} = X[1-\tau]$, the kernel of $1-\tau$, is a $\Lambda$-module. It is the inverse limit
of the groups $A_n^{\tau}$. A natural question
is whether $X/X^{\tau}$ is finite. Our calculations indicate that there are many examples where this is the case
when $p=3$. However, when $p=2$, in several cases it appears that $\lambda$ is positive; if so,
$X/X^{\tau}$ is not finite.

Another way of looking at this is the following. Let $X_0$ be the $\mathbb Z_p$-torsion submodule of $X$.
It is a $\Lambda$-module. Then $X/X_0$ is a free $\mathbb Z_p$-module of rank $\lambda$. Since the kernel
of $1-\tau$ on $X/X_0$ has exponent dividing $p$, this kernel is trivial. In other words, the ambiguous
subgroup lives in the $\mu$ part of $X$, namely $X_0$, and has no obvious effect on $\lambda$.

We now consider $K_n/k_n$.
The same estimates as in the proof of Theorem 1 show that the number of strongly ambiguous classes
of $K_n$ is at least $p^{(s-1)p^n-1}$ and at most $p^{sp^n-1}$.
Therefore, the strongly ambiguous classes contribute either $s-1$ or $s$ to $\mu$.
Proposition \ref{ambig} implies that the ambiguous classes mod the strong classes contribute 0 or 1 to $\mu$.
The total of these two contributions to $\mu$ is $s-1$ or $s$. 
Note that if $\mu=s-1$, then all of the contribution to $\mu$ is from the strongly ambiguous classes.

\begin{proposition}\label{weakambig} Assume that all primes above $p$ are totally ramified in $K_{\infty}/K_0$. Suppose $\mu=s-1$. Then 
the index of the group of strongly ambiguous classes of $A_n$ in the
$p$-torsion of $A_n$ is bounded independent of $n$. Therefore, the order of the group of ambiguous classes in $A_n$ mod the strongly ambiguous
classes is bounded independent of $n$.
\end{proposition}
\begin{proof}
The following lemma is useful.
\begin{lemma}\label{fourterm} Let $0\to T_n\to U_n \to V_n\to W_n\to 0$, for $n=0, 1, 2, 3, \dots$ be exact sequences of finite abelian groups
and suppose the orders of $T_n$ and $W_n$ are bounded independent of $n$. Let $p$ be a prime and let $U_n[p]$ and $V_n[p]$ denote
the $p$-torsion of $U_n$ and $V_n$. Then $\#U_n[p] / \#V_n[p]$ is bounded above and away from 0, independent of $n$.
\end{lemma}
\begin{proof}
The exact sequence $0\to U_n/T_n \to V_n\to W_n\to 0$ yields, via the Snake Lemma, the exact sequence
$$
0\to (U_n/T_n)[p| \to V_n[p] \to W_n[p].
$$
Since $\#W_n[p]$ is bounded by the order of $W_n$, we find that $\#(U_n/T_n)[p] / \#V_n[p]$ is bounded independent of $n$
(here, and in the rest of the proof of the proposition, we use ``bounded'' to mean bounded both above and away from 0).
The exact sequence $0\to T_n\to U_n\to U_n/T_n\to 0$ yields the exact sequence
$$
T_n[p] \to U_n[p] \to (U_n/T_n)[p] \to T_n/pT_n,
$$
which implies that $\#U_n[p]/\#(U_n/T_n)[p]$ is bounded independent of $n$. 
This yields the lemma. \end{proof}

As in Section 6, there is a submodule $Y_0\subseteq X$ such that $X/\nu_nY_0\simeq A_n$. There is an elementary $\Lambda$-module
$E$ and an exact sequence
$$
0\to F_1\to Y_0 \to E \to F_2\to 0
$$
with $F_1, F_2$ finite. An argument similar to the proof of Lemma \ref{fourterm} implies that
we have exact sequences
$$
0\to T_n\to Y_0/\nu_nY_0 \to E/\nu_nE \to W_n\to 0
$$
for $n=0, 1,2, \dots$, with the orders of $T_n$ and $W_n$ bounded independent of $n$.
Therefore, $\#(Y_0/\nu_nY_0)[p]/\#(E/\nu_nE)[p]$ is bounded independent of $n$.

We also have the exact sequence $0\to Y_0/\nu_nY_0\to X/\nu_nY_0\to X/Y_0\to 0$ with $X/Y_0$ finite.
Since $A_n\simeq X/\nu_nY_0$, we obtain that
$\#A_n[p]/\#(Y_0/\nu_nY_0)[p]$ is bounded independent of $n$.
Therefore $\#A_n[p]/\#(E/\nu_nE)[p]$ is bounded independent of $n$.

Consider a summand $\Lambda/(p^i)$ of $E$. Since $\Lambda/(p^i, \nu_n)$ is finite,
$$
\#(\Lambda/(p^i, \nu_n))[p] = \#(\Lambda/(p^i, \nu_n,p)) = p^{p^n-1}.
$$
The terms of the form $\Lambda/(p^i)$ therefore contribute at most $p^{\mu (p^n-1)}$ to $\#(E/\nu_nE)[p]$, with equality
if and only if each $i$ is 1. 

Now consider a summand $\Lambda/(f)$, where $f$ is a distinguished polynomial. Let $p^n-1\ge \deg f$. Since
$E/\nu_nE$ is finite, $\Lambda/(f, \nu_n)$ is finite. Therefore,
$$
\#(\Lambda/(f, \nu_n))[p] = \#(\Lambda/(f, \nu_n, p))=p^{\deg f}.
$$
The terms of the form $\Lambda/(f)$ therefore contribute $p^{\lambda}$ to $\#(E/\nu_nE)[p]$ when $n$ is sufficiently large.

Putting these together, we find that
$$
\#(E/\nu_nE)[p]/p^{\mu p^n +\lambda}
$$
is bounded above, independent of $n$. 
Therefore,
$$
\#A_n[p]/p^{\mu p^n+\lambda}
$$
is bounded above.

We know that the ambiguous classes in $A_n$ are contained in $A_n[p]$.
Since $[E_{k_n} : N(E_{K_n})] \le p^{p^n}$, Corollary \ref{strongambig} implies that 
$$
p^{(s-1)p^n-1}\le \frac{\prod_{\mathcal P \in S} e_P}{p [E_{k_n}  : N(E_{K_n})]} \le \#A_n[p] \le Cp^{\mu p^n+\lambda}
$$
for some constant $C$, independent of $n$. If $\mu=s-1$, the ratio of the fourth term to the first term is bounded.
Therefore, the ratio of the third term to the second term must be bounded. This completes the proof
of Proposition \ref{weakambig}. \end{proof}

A slight extension of the method of the proof also shows that if $\mu=s-1$, then all summands of $E$ of the form $\Lambda/(p^i)$ have $i=1$. 

An interesting question is whether there can be a contribution from the ambiguous classes mod the strong classes if $\mu=s$.
Also, if $\mu=s$, must there be a contribution from the ambiguous classes modulo the strong classes?

The following shows that in the cases we considered for $p=3$, except for $d=17\cdot 53\cdot 71$, all of the ambiguous classes at the base level $K_0$
are strong. 

\begin{proposition}\label{allstrong3} Let $k_0=\mathbb Q(\sqrt{-3})$ and $K_0=k_0(\sqrt[3]{d})$, where $d$
is cube-free and contains a prime factor $q\ne 3$ with  $q\not\equiv \pm 1\pmod 9$. Then all of the ambiguous classes of $K_0$
are strong.
\end{proposition}
\begin{proof}  By Lemma \ref{ambig}, we must show that $E_{k_0}\cap N(K_0^{\times})=N(E_{K_0})$.
In fact, we'll show that $\zeta_3\not\in N(K_0^{\times})$, so both sides are $\{\pm 1\}$.
Complete $K_0$ and $k_0$ at primes above $q$.
Since $K_0/k_0$ is ramified at $q$, local class field theory says that the norms of the local units of $K_0$
form a subgroup of index 3 in the local units of $k_0$. Since $q^2- 1\not\equiv 0\pmod 9$, the local units
do not have a subgroup of index 9, so $\zeta_3$ cannot lie in a subgroup of index 3. Therefore, $\zeta_3$ is
not a local norm, hence not a global norm.\end{proof}

If, in addition, $q\equiv 2, 5 \pmod 9$, then $q$ splits completely in $k_n/k_0$. Therefore, for each $n$, 
the proof of the proposition
shows that $\zeta_3$ is not a norm for $K_n/k_n$. However, there are many more units of $k_n$, so
it is not clear how many non-strongly ambiguous classes there are for $K_n$. 

In the proposition, some assumption on $q$ is needed. For example, when $d=51$, the primes of $K_0$ above $3$ and $17$
are principal, so the strongly ambiguous classes are trivial. However, the  class group of $K_0$
is $3\times 3$. Since an element of order 3,  for example $\tau$, acting on a 3-group must have a nontrivial
fixed point, there is a nontrivial ambiguous class that is not strong. Lemma \ref{ambig} implies that
the ambiguous classes mod the strongly ambiguous classes have order at most 3, so the ambiguous classes
have order 3 in this case.

When $p=2$, we have the following.
\begin{proposition}\label{allstrong2}  Let $k_0=\mathbb Q(i)$ and $K_0=k_0(\sqrt{d})$, where $d$
is odd and squarefree and contains a prime factor $q$ with  $q\equiv 5\pmod 8$. Then all of the ambiguous classes of $K_0$
are strong. 
\end{proposition}
\begin{proof} As in the proof of Proposition \ref{allstrong3}, we complete $K_0$ and $k_0$ at primes above $q$
and find that $i$ is not a local norm. Therefore, $i$ is not a global norm, so $E_{k_0}\cap N(K_0^{\times}) = N(E_{K_0}^{\times})= \{\pm 1\}$.
\end{proof}

If $d$ has no prime factor $q\equiv 5\pmod 8$, then $i$ is everywhere a local norm, hence a global norm.
Whether it is a norm of a unit is a harder question. 

\medskip

{\bf Data.} Since $k_n$ has class number prime to $p$ for each $n$, the strongly ambiguous classes for $K_n/k_n$ are generated by the ramified primes
for this extension. Using Pari-GP, we calculated the contributions of these ideals to the class group of $K_n$ for each of the examples
in Section \ref{SectComp}. Corollary \ref{strongambig} then allowed the computation of $[E_{k_n} : N(E_{K_n})]$. 
The results for $p=2$ and for $p=3$ differ in nature.

For $p=2$ and $d\ne 3\cdot 11\cdot 19$, the strongly ambiguous classes have index 2 in $A_n[2]$ in all of the examples, and
the index $[E_{k_n} \, : \, N(E_{K_n})]$ is maximal. That is, the index is $2^{2^n}$, which is also the index of $E_{k_n}^2$ 
in $N(E_{K_n})$. 

In the exceptional case $d=3\cdot 11\cdot 19$,  the group of strongly ambiguous classes equals $A_n[2]$ for
$K_0$ and $K_1$, while the index is 2 for $K_2$ and $K_3$. Correspondingly, $[E_{k_n} \, : \, N(E_{K_n})]$ equals 1 for
$n=0$, equals 2 for $n=1$, and is $2^{2^n}$ for $n=2$ and $n=3$. The index is 1 for $n=0$ because $i\in N(E_{K_0})$.
However, $i=\zeta_8^2\in E_{k_1}^2$, so this does not affect the index for $n=1$. For $n=1$, a calculation shows that
$\zeta_8(1+\sqrt{2})\in N(E_{K_1})$, which causes the index to be 2 instead of 4. However, $\zeta_8(1+\sqrt{2})\in E_{k_2}^2$,
so this does not affect the index for $n=2$. By Proposition \ref{sqrtunit}, this could potentially go on forever.
However, the index is maximal for $n=2$ and $n=3$, after which the computations become too lengthy.

For $p=3$, the index of the strongly ambiguous classes in $A_n[3]$ varies. For example, for $d= 10$, $44$, $46$, and $253$, 
the group of strongly ambiguous classes
equals $A_n[3]$ for $n=0$ and $1$. For $d=22$, $34$, $170$, and several others, they have index $3$ for $n=0$ and $1$.
For $d=94$, they have index 3 for $n=0$ and index 27 for $n=1$. So far, we have not found any predictable behavior analogous to what happens for $p=2$.
However, in all examples from Section \ref{SectComp} with $d\ne 17\cdot 53\cdot 71$, the index $[E_{k_n} \, : \, N(E_{K_n})]$ is maximal.
That is, it equals $3^{3^n}$. For the exceptional $d= 17\cdot 53\cdot 71$, the index is $1$ for $n=0$ and $3^2$ for $n=1$.
This is caused by the fact that $\zeta_3\in N(K_0)$, and hence $\zeta_3\in N(K_n)$ for all $n\ge 0$.

We know that 
$$1\le [E_{k_n} : N(E_{K_n})]\le p^{p^n},$$
and therefore
$$
p^{(s-1)p^n-1}\le \frac{\prod_{\mathcal P \in S} e_P}{p [E_k  : N(E_K)]}\le p^{sp^n-1}.$$
The upper bound corresponds to the classes for the ramified primes being independent in $A_n[p]$, except for the relation that their product is $\sqrt[p]{d}$.
However, the evidence suggest that the unit index is equal to, or very close to, its upper bound, in which case the number of strongly ambiguous classes is at the lower bound.
This means that we should expect the strongly ambiguous classes to contribute exactly $s-1$ to $\mu$. If $\mu > s-1$, there must be another source of ideal classes.
When $p=2$, the index of the strongly ambiguous classes in $A_n[2]$ is usually 2, so it is reasonable to 
guess that it is bounded as $n$ increases. Therefore, it seems unlikely
that the remaining, non-strongly ambiguous classes contribute to $\mu$.
This is perhaps also the case when $p=3$, although the evidence is not as strong. The largest index 
among our examples is $3^4$ for $d=17\cdot 53\cdot 71$ and $n=1$,
where $A_1[3]= 3^{10}$ and the group of strong classes is $3^6$. 

There is a filtration
$$
 A_n[1-\tau] \subseteq A_n[(1-\tau)^2] \subseteq A_n[(1-\tau)^3] \subseteq \cdots \subseteq A_n.
$$
The $k$th {\it higher ambiguous group} is defined to be 
$$A_n[(1-\tau)^k]/A_n[(1-\tau)^{k-1}].$$ 
Since $k_n$ has class number prime to $p$, the norm $1+\tau+\tau^2+\cdots + \tau^{p-1}$ annihilates $A_n$,
so $A_n$ is a module over $\mathbb Z[\zeta_p]$, where $\zeta_p$ acts as $\tau$. Therefore, 
$$A_n[(1-\tau)^{m(p-1)}]=A_n[p^m].$$

For $p=2$, it follows that the 2-rank of the $k$th higher ambiguous group is the number of summands of $A_n$ of the form $2^j$ with $j\ge k$.
Since a term $\Lambda/(4)$ in the elementary $\Lambda$-module associated to $\lim_{\leftarrow} A_n$ corresponds to many summands of 4 in $A_n$, and also
since such a summand makes $\mu$ larger than anticipated, it is possible that the higher ambiguous groups can be used to explain cases where
$\mu > s-1$. What is needed is a formula for the order of the higher ambiguous groups that is analogous to Proposition \ref{Chev}.

For $p=3$, the first and second higher ambiguous groups come from the 3-torsion in $A_n$, so the discrepancy between
the strongly ambiguous classes, which are all the ambiguous classes in many situations by Proposition \ref{allstrong3}, and the 3-torsion $A_n[3]$,
is the 2nd higher ambiguous group. Therefore, this second group measures the fluctuations in the index of the strong
classes in the group $A_n[3]$ that was observed in the data in Section \ref{SectComp} for $p=3$, in contrast to the case $p=2$.
Again, these higher ambiguous groups might yield an explanation of cases where $\mu>s-1$.

\section{$\mathbf{s=1}$}

The case where $s=1$ exhibits some interesting phenomena that are not seen in the previous examples. In this case, we have $\mu\ge s-1=0$, 
so we get no information on $\mu$
from Theorem \ref{Iw}. 

First let's consider the case $p=2$. The base field is $K_0=\mathbb Q(i, \sqrt{q})$ with a prime $q\equiv 3\pmod 4$. Since $q$ is the only prime ramifying in $K_0/k_0$, Proposition \ref{Chev} tells us that
$$
\# A_0[1-\tau] = \frac{2}{2 [\langle i\rangle : N(E_{K_0})]},
$$
so $A_0[1-\tau]=1$. When $\tau$, which has order 2, acts on a non-trivial 2-group, it has a non-trivial fixed point. Therefore, $A_0=1$. 

Computations indicate that 
\begin{itemize}
\item $A_n=1$ for all $n\ge 0$  when $q\equiv 3\pmod 8$.
\item $A_n=2$ for all $n\ge 1$ when $q\equiv 7\pmod {16}$.
\item $A_1\supseteq 4$ and $A_2\supseteq 4\times 2^2$ when $q\equiv 15\pmod{16}$.
\end{itemize}
The observations for $q\equiv 3\pmod 8$ can be proved as follows: The prime above 2 is inert in $K_0/k_0$, so $K_{\infty}/K_0$ is ramified at only one prime and is totally ramified.
Since $A_0=1$, it follows that $A_n=1$ for all $n$.

We have not yet found proofs of the statements for $q\not\equiv 3\pmod 8$.

Here are some examples:

\noindent
$\mathbf{q=7, 23, 71:}$ For these primes, $A_1=2$ and $A_2=2$. Since $e_2-e_1=0$, it follows that $A_n=2$ for all $n\ge 1$.
Therefore,
$$
\mu=0, \quad \lambda =0, \quad \nu=0.
$$
 Since $\tau$ must act trivially on a group of order 2,
 the strong classes are trivial and therefore have index
$2$ in $A_n[2]$ for $n\ge 1$. Proposition \ref{Chev} implies that the index $[E_{k_n} : N(E_{K_n})]$ is $2^{2^n-1}$ for $n\ge 0$, 
which is the half the largest possible value.

\medskip
\noindent
$\mathbf{q=31:}$ We have $A_1=4$, $A_2=8\times 4\times 2$, and $A_3=16\times 8\times 2$. Since $e_1-e_0=2$, we have $0\le \mu\le 2$. 
The structure of the groups suggests
$\lambda =2$ and $\mu=0$. The strong classes are trivial for $n=1, 2$, 
and have order $2^2$ for $n=3$. This implies that $[E_{k_n} : N(E_{K_n})]=2^5$ for $n=3$, which is one-eighth of the largest possible value.

\medskip
\noindent
$\mathbf{q=47:}$ We have $A_1=4$, $A_2=4\times 2^2$, and $A_3= 8^3\times 4\times 2$. Since $e_1-e_0=2$, we have $0\le \mu\le 2$. 
The increasing ranks of the groups suggests the possibility of a positive $\mu$.
 The strong classes are trivial for $n=1, 2$, and have order $2^2$ for $n=3$. This implies 
that $[E_{k_n} : N(E_{K_n})]=2^5$ for $n=3$, which is one-eighth of the largest possible value. The strong classes have the rather large index
$2^3$ in $A_3[2]$. 

\medskip
\noindent
$\mathbf{q=79:}$ We have $A_1=4$, $A_2=16\times 8\times 2$, and $A_3= 16\times 8\times 2$. Since $e_3-e_2=0$, we have 
$$
\mu=0, \quad \lambda =0, \quad \nu=8.
$$
 The strong classes are trivial for $n=1$, and have order $2^2$ for $n=2, 3$. This implies that $[E_{k_n} : N(E_{K_n})]=2^5$ for $n=3$, 
which is one-eighth of the largest possible value. The strong classes have index
$2^2$ in $A_3[2]$. This example is somewhat distressing since it highlights the danger of 
guessing $\mu$ and $\lambda$ from early data. The groups $A_1$ and $A_2$ suggest a positive $\lambda$ or $\mu$.
So it was surprising when $A_3$ had the same order as $A_2$.

\medskip

We now consider $p=3$, so the base field is $K_0=\mathbb Q(\sqrt{-3}, \sqrt[3]{q})$, with a prime $q\equiv 2\pmod 3$.

Two of the cases yield $\mu=\lambda=\nu=0$.
\begin{proposition} Suppose $q\equiv 2 \text{ or } 5\pmod 9$ is prime. 
Then the class number of $\mathbb Q(\zeta_3, \sqrt[3]{q})$ is not a multiple of 3.
\end{proposition}
\begin{proof} Since $q$ is inert in $\mathbb Q(\zeta_3)$, the local units in $\mathbb Z_q[\zeta_3]$ are the product of a cyclic group of
order $q^2-1$ and a profinite $q$-group. Since $q^2-1\not\equiv 0\pmod 9$, there is only one subgroup of index 3 and it does not contain
$\zeta_3$. Since $q$ is ramified in $\mathbb Q(\zeta_3, \sqrt[3]{q})/\mathbb Q(\zeta_3)$, the local norms from the completion of
this extension are powers of $q$ times elements from a subgroup of index 3 in the local units. In  particular, $\zeta_3$ is not a local
norm, hence is not a global norm. In Proposition \ref{Chev} with $K=\mathbb Q(\zeta_3)$ and $L=\mathbb Q(\zeta_3, \sqrt[3]{q})$,
the unit group $E$ is generated by $\zeta_3$, so $[E : E\cap N(L^{\times})] = 3$. The product of the ramification indices is $3^2$.
Therefore, $A_0^{\tau}$ has order 1. But a 3-group acting on a non-trivial 3-group always has a non-trivial fixed point. 
If $A_0$ had a non-trivial 3-subgroup
then $A_0^{\tau}$ would be non-trivial. Therefore, the class number of $\mathbb Q(\zeta_3, \sqrt[3]{q})$ is not a multiple of 3. \end{proof}

\begin{corollary} Let $k_{\infty}/k_0$ be the anticyclotomic $\mathbb Z_3$-extension of $k_0=\mathbb Q(\sqrt{-3})$. Let $q\equiv 2 \text{ or } 5\pmod 9$ be prime, let $K_0= \mathbb Q(\zeta_3, \sqrt[3]{q})$,
and let $K_n=k_n(\sqrt[3]{q})$. Then the class number of $K_n$ is prime to 3 for all $n\ge 0$.
In particular, $\mu=\lambda=\nu=0$ for the $\mathbb Z_3$-extension $K_{\infty}/K$.
\end{corollary}
\begin{proof} Since there is only one prime ramified in the 3-extension $K_n/K$ and 3 does not divide the class number of $K$,
it follows that 3 does not divide the class number of $K_n$. \end{proof}

We now consider $q\equiv 8\pmod 9$. Then the prime above 3 is unramified in $K_n/k_n$ for all $n\ge 0$, so the primes above $q$ are the only primes that ramify in $K_n/k_n$. Proposition \ref{Chev}
implies that $A_0=1$, just as in the case $p=2$ considered above. Only $K_1$ has degree small enough to perform calculations, but we can make guess for $K_2$ based on the ideas of Section \ref{conj}. The results are the following
(we have not yet discovered the rule for dividing the primes $q$ into these sets corresponding to the various possibilities for $A_1$):

\medskip
\noindent
$\mathbf{q= 17, 53, 71, 107, 179, 197, 233, 251, 359, 467, 503, 557, 701, 863, 881, 953, 971}$: We have $A_1=9\times 3$. The strong classes have order 3 and have index 3 in $A_1[3]$.
The unit index is 3. We have $e_1-e_0=3$, so $0\le \mu \le 1$, but it is not clear how to decide which is the correct value.

In the notation Section \ref{conj}, we have $3^2\mid\mid h_2'$ for each of these values of $d$. It seems unlikely that $h_2=(h_2')^2$ since these are primes that are $-1\pmod 9$, but suppose
that $h_2\mid 3^4 (h_2')^2$. Then $e_2\le 8$. If we take the base field of the $\mathbb Z_3$-extension as $K_1$, then $\mu$ becomes $3\mu$, $e_2$ becomes $e_1\le 8$, and $e_1$ becomes $e_0=3$.
We have the inequality $8-3\ge e_1-e_0\ge 3\mu (3^1-1)$, which yields $\mu=0$.

\medskip
\noindent
$\mathbf{q= 89, 431, 449, 593, 647, 683, 719, 773}$: We have $A_1=3^2$. The strong classes are trivial and the unit index is therefore $3^2$.  Again, we have $0\le \mu\le 1$.

For $d=89$ and $449$, we have $3^5\mid\mid h_2'$. We would need $h_2\mid (1/27) (h_2')^2$ for our inequalities to imply that $\mu=0$.  

For $d=431, 647, 719$, we have $3^2\mid\mid h_2'$. If $h_2\mid 3^3 (h_2')^2$, we can deduce that $\mu=0$. 

For $d=683$ and $773$, we have $3^4\mid\mid h_2$. We would need $h_2\mid (1/3) (h_2')^2$ for our inequalities to imply that $\mu=0$.  

For $d=593$, the computation of $h_2'$ took too long, so we do not have its value.

\medskip
\noindent
$\mathbf{q=269, 521, 809}$: We have $A_1=9^2\times 3^2$. The strong classes are trivial and the unit index is therefore $3^2$. We have $0\le \mu\le 3$.

For $d=269$ and $809$, we have $3^6\mid\mid h_2'$, and for $d=521$ we have $3^9\mid\mid h_2'$. It seems likely that $\mu\le 1$ for $269$ and $809$, and $\mu\le 2$
for $521$.

\medskip
\noindent
$\mathbf{q=827}$: We have $A_1= 27\times 9\times 3^2$. The strong classes have order 3 and the unit index is therefore $3$. We have $0\le \mu \le 3$.

We have $3^6\mid\mid h_2'$. If $h_2\mid (h_2')^2$, our inequalities, applied with $K_1$ as the base of the $\mathbb Z_3$-extension as in the cases of $q=17, 53, \dots$ above, 
 yield $\mu=0$.

\end{document}